\UseRawInputEncoding 
\documentclass[titlepage]{article}
\usepackage{layout}
\usepackage{url}
\usepackage{proof}
\usepackage{mathrsfs}
\usepackage{amsmath,amsthm,amssymb}
\usepackage{braket}
\usepackage{bussproofs}
\usepackage[dvipdfmx]{graphicx}%%%
\usepackage[dvipdfmx]{color}
\usepackage{mathtools}
\usepackage[all]{xy}

\theoremstyle{definition}
\newtheorem{thm}{Theorem}
\newtheorem{prop}[thm]{Proposition}
\newtheorem{lemma}[thm]{Lemma}

\newtheorem{defn}[thm]{Definition}
\newtheorem{cor}[thm]{Corollary}

\newtheorem{rmk}[thm]{Remark}

\newtheorem{convention}[thm]{Convention}

\newcommand{\NN}{\mathbb{N}}
\newcommand{\ZZ}{\mathbb{Z}}

\newcommand{\deff}[1]{\textbf{\textit{#1}}}

\DeclareMathOperator{\dom}{dom}
\DeclareMathOperator{\ran}{ran}

\title{A note on the independence of the injective pigeonhole principle from the uniform counting principle}
\author{Eitetsu Ken\footnote{email: yeongcheol.e.k@gmail.com}}
\begin{document}
\maketitle

\begin{abstract}
In this technical note, we show that the relativized bounded arithmetic $T^{1}_{2}(R)$ equipped with $UCP(\Delta^{b}_{1}(R))$ cannot prove $injPHP^{n+1}_{n}(R)$.
The result partially supports \cite{UCPvsinjPHPrevised} Conjecture 1.
%Furthermore, we show that any proofs of $injPHP^{m}_{n}$ in the second level of constant-depth Frege system needs exponential growth-rate with respect to $n$.
\end{abstract}

\section{Introduction}
This work is a direct continuation of \cite{UCPvsinjPHPrevised}.
We assume that the reader is familiar with the content of \cite{UCPvsinjPHPrevised} up to \S 3.
There, the \textit{Uniform Counting Principle} ($UCP$) is introduced to consider a possible natural generalization of the result of \cite{More}, namely, $AC^{0}$-Frege system equipped with the modular counting principle $Count^{p}_{k}$ for an arbitrary fixed modulus $p$ as an axiom scheme needs exponentially large proofs to prove the injective pigeonhole principle $injPHP^{n+1}_{n}$.
\cite{UCPvsinjPHPrevised} observed that $UCP$ implies any instance of $Count^{p}_{k}$ in $AC^{0}$-Frege system with short proofs and conjectured (casted as Question 1) that $AC^{0}$-Frege system equipped with $UCP^{l,d}_{k}$ as an axiom scheme still needs exponentially large proofs to prove $injPHP^{n+1}_{n}$.

In this note, we show that the statement holds at least for the depth $1+\frac{1}{2}$-fragment of $AC^{0}$-Frege system.
In terms of bounded arithmetics for the robustness, we show that $T^{1}_{2}(R)$ equipped with the axiom scheme $UCP^{l,d}_{k}(\varphi)$ ($\varphi$ is a $\Delta^{b}_{1}(R)$-formula in $S^{1}_{2}(R)$), denoted by $T^{1}_{2}(R)+UCP^{l,d}_{k}(\Delta^{b}_{1}(R))$, does not prove $injPHP^{n+1}_{n}(R)$.
As a corollary, we also obtain 
\[T^{1}_{2}(R)+\forall M, m.\ ontoPHP^{M}_{m}(\Delta^{b_{1}}(R)) \not \vdash \forall n.\ injPHP^{n+1}_{n}(R),\]
where $ontoPHP^{M}_{m}(\varphi)$ is a natural formalization of the statement ``If $M>m$, then $\varphi$ does not code a bijection between $[M]$ and $[m]$.'' 
The comparison between the two types of the pigeonhole principles $ontoPHP$ and $injPHP$ is interesting in its own right since it is related to the complexity of proofs of Cantor-Bernstein theorem.

Unless stated otherwise, we follow the convention and the notation of \cite{UCPvsinjPHPrevised} throughout this note.

The article is organized as follows.

In \S \ref{A proof theoretic approach}, we present a particular applied propositional sequent calculus $LK^{*}_{1+\frac{1}{2},c}(UCP)$, which is a propositional translation of the bounded arithmetic $T^{1}_{2}(R) + UCP^{l,d}_{k}(\Delta^{b}_{1}(R))$.

In \S \ref{evaluation}, we give a variation of Switching Lemma, tailored for formulae in $LK^{*}_{1+\frac{1}{2},O(1)}(UCP)$-proofs of the injective pigeonhole principle $injPHP^{n+1}_{n}$.
A tricky point here is that we can avoid having singletons in each partial injections corresponding to branches of $injPHP$-trees assigned to formulae of depth $0+\frac{1}{2}$.
This allows us to carry out an argument analogous to the one for Ajtai's theorem presented in \cite{remake} and \cite{remake2}. (See also \cite{proofcomplexity} Chapter 15 for a neat exposition with historical backgrounds.)

In \S \ref{The Main Result}, we show our main result: the independence of $injPHP^{n+1}_{n}(R)$ from $UCP^{l,d}_{k}(\Delta^{b}_{1}(R))$ in $T^{1}_{2}(R)$, following the proof strategy presented in \cite{UCPvsinjPHPrevised}. 
The argument is a straightforward modification of the proof of Theorem 53 in \cite{UCPvsinjPHPrevised}.
The proof is made somewhat simple by expanding $UCP^{l,d}_{k}$ into boolean combinations of constantly many depth $1+\frac{1}{2}$-formulae in the formulation of $LK_{1+\frac{1}{2},O(1)}^{*}(UCP)$ given in \S \ref{Preliminaries}.

In \S \ref{A model theoretic approach}, we give another proof of $T^{1}_{2}(R)+UCP^{l,d}_{k}(\Delta^{b}_{1}(R)) \not \vdash injPHP^{n+1}_{n}(R)$, following the framework in \cite{partiallydefinableforcing}.
Actually, the proof can be slightly changed to show the independence of much weaker variant of the injective pigeonhoke principle:
\[T^{1}_{2}(R) + UCP^{l,d}_{k}(\Delta^{b}_{1}(R)) \not\vdash injPHP^{univ}_{n}(R),\]
 where $injPHP^{univ}_{n}(R)$ states that ``$R$ cannot code an injection from the whole universe into $[n]$'' (Remark \ref{mostweakPHPisindependent}).

\section{Preliminaries}\label{Preliminaries}

\textit{Uniform Counting Principle} is defined as follows:
\begin{defn}\label{UCP}
$UCP(l,d,n,R)$ (which stands for \textit{Uniform Counting Principle}) is an $\mathcal{L}^{2}_{A}$ formula defined as follows (note that the definition of $UCP$ in \cite{UCPvsinjPHPrevised} has a typo on bracketing):
\begin{align*}
(d \geq 1 \land \lnot d \mid n ) \rightarrow \lnot
&\bigg( \forall i \in [l]. (\forall j\in [d]. \exists e \in [n]. R(i,j,e)  \lor \forall j \in [d]. \lnot \exists e \in [n]. R( i,j,e )) \\
&\land \forall (i,j) \in [l]\times[d]. \forall e \neq e^{\prime} \in [n] (\lnot R( i,j,e) \lor \lnot R( i,j,e^{\prime}))\\ 
&\land \forall (i,j) \neq (i^{\prime},j^{\prime}) \in [l] \times [d]. \forall e \in [n]. (\lnot R(i,j,e ) \lor \lnot R( i^{\prime},j^{\prime},e ))\\
&\land \forall e \in [n]. \exists (i,j) \in [l] \times [d]. R( i,j,e ) \bigg)
\end{align*}
%Here, $\langle \cdot \rangle$ denotes a $\Sigma^{B}_{0}$-definable tupling function which is monotone for each input. 

With some trivial manipulation on the quantifiers, it can be equivalently formulated as follows:
\begin{align*}
(d \geq 1 \land \lnot d \mid n ) \rightarrow \lnot
&\bigg( \forall i \in [l]. \forall j\in [d]. ( \exists e \in [n]. R(i,j,e)  \lor \forall j' \in [d]. \forall e' \in [n]. \lnot R( i,j',e' )) \\
&\land \forall (i,j) \in [l]\times[d]. \forall e \neq e^{\prime} \in [n] (\lnot R( i,j,e) \lor \lnot R( i,j,e^{\prime}))\\ 
&\land \forall (i,j) \neq (i^{\prime},j^{\prime}) \in [l] \times [d]. \forall e \in [n]. (\lnot R(i,j,e ) \lor \lnot R( i^{\prime},j^{\prime},e ))\\
&\land \forall e \in [n]. \exists (i,j) \in [l] \times [d]. R( i,j,e ) \bigg)
\end{align*}

The propositional formula $UCP^{l,d}_{n}$ is defined as follows:
\begin{align*}
UCP^{l,d}_{n}:=
\begin{cases}
\lnot \bigg(\bigwedge_{i=1}^{l}\left( \left(\bigwedge_{j=1}^{d} \bigvee_{e \in [n]} r_{i,j,e} \right) \lor \left( \bigwedge_{j=1}^{d} \lnot \bigvee_{e \in [n]} r_{i,j,e} \right)\right) \\
\land \bigwedge_{(i,j) \in [l] \times [d]} \bigwedge_{e \neq e^{\prime} \in [n]} (\lnot r_{i,j,e} \lor \lnot r_{i,j,e^{\prime}})\\ 
\land \bigwedge_{(i,j) \neq (i^{\prime},j^{\prime}) \in [l] \times [d]} \bigwedge_{e \in [n]} (\lnot r_{i,j,e} \lor \lnot r_{i^{\prime},j^{\prime},e})\\
\land \bigwedge_{e \in [n]} \bigvee_{(i,j) \in [l] \times [d]} r_{i,j,e} \bigg) 
\quad (\mbox{if $n \not \equiv 0 \pmod d$, $d \geq 1$})\\
1 \quad (\mbox{otherwise})
\end{cases}
\end{align*}
We abuse the notation and use $UCP^{l,d}_{n}$ to express $UCP(l,d,n,R)$. (See also \cite{UCPvsinjPHPrevised} Convention 3.)
\end{defn}

\section{A proof theoretic approach}\label{A proof theoretic approach}
The theory $T^{1}_{2}(R)+UCP^{l,d}_{k}(\Delta^{b}_{1}(R))$ can be expressed as an applied sequent calculus dealing with formulae of negation normal form: (we follow the convention and treatment of one-sided sequent caluculi in \cite{OrdinalAnalysis} \S 2.1.2)
\begin{defn}
Let $\bold{G}+T^{1}_{2}(R)+UCP^{l,d}_{k}$ be an applied sequent calculus whose proof tree obeys the following derivation rules:  
 \begin{itemize}
 \item Initial Sequent;
 
 \begin{prooftree}
 \AxiomC{}  \RightLabel{\quad (where $L$ is a literal)}
  \UnaryInfC{$\Gamma, L, \overline{L}$}

\end{prooftree}

 \item $\lor$-Rule; 
   \begin{prooftree}
 \AxiomC{$\Gamma, \varphi_{i_{0}}$} \RightLabel{\quad (where $\varphi_{1} \lor \varphi_{2} \in \Gamma$, $i_{0}=1,2$)}
  \UnaryInfC{$\Gamma$}
 \end{prooftree}

   \item $\exists$-Rule: 
   \begin{prooftree}
 \AxiomC{$\Gamma, \varphi(u)$}
  \RightLabel{\quad (where $\exists x. \varphi(x) \in \Gamma$)}
  \UnaryInfC{$\Gamma$}
 \end{prooftree}

 \item $\exists^{\leq}$-Rule; 
   \begin{prooftree}
 \AxiomC{$\Gamma, \varphi(u)$}
 \AxiomC{$\Gamma, u \leq t$}
  \RightLabel{\quad (where $\exists x \leq t. \varphi(x) \in \Gamma$, and $u$ is a term)}
  \BinaryInfC{$\Gamma$}
 \end{prooftree}
 
 \item $\land$-Rule;
    \begin{prooftree}
 \AxiomC{$\Gamma, \varphi_{1}$}
 \AxiomC{$\Gamma, \varphi_{2}$}
   \RightLabel{\quad (where $\varphi_{1} \land \varphi_{2} \in \Gamma$)}
  \BinaryInfC{$\Gamma$}
 \end{prooftree}

 \item $\forall$-Rule:
    \begin{prooftree}
 \AxiomC{$\Gamma, \varphi(a)$}
   \RightLabel{\quad (where $\forall x. \varphi(x) \in \Gamma$, and $a$ is an eigenvariable)}
  \UnaryInfC{$\Gamma$}
 \end{prooftree}

 \item $\forall^{\leq}$-Rule;
    \begin{prooftree}
 \AxiomC{$\Gamma, \overline{a \leq t}, \varphi(a)$}
   \RightLabel{\quad (where $\forall x \leq t. \varphi(x) \in \Gamma$, and $a$ is an eigenvariable)}
  \UnaryInfC{$\Gamma$}
 \end{prooftree}

 \item Axiom of $T^{1}_{2}(R)$;
  \begin{prooftree}
 \AxiomC{$\Gamma, \overline{\varphi}$} \RightLabel{\quad (where $\varphi$ is a substitution instance of one of open axioms of $T^{1}_{2}(R)$)}
  \UnaryInfC{$\Gamma$}
 \end{prooftree}

 \item $\Sigma^{b}_{1}(R)$-Induction;
     \begin{prooftree}
 \AxiomC{$\Gamma, \varphi(0)$}
 \AxiomC{$\Gamma, \overline{\varphi(a)},\varphi(a+1)$}
  \AxiomC{$\Gamma, \overline{\varphi(t)}$}
  \RightLabel{\quad (where $\varphi \in \Sigma^{b}_{1}(R)$, $a$ is an eigenvariable)}
  \TrinaryInfC{$\Gamma$}
 \end{prooftree}
 
 \item $UCP$;
 
deriving $\Gamma$ from the following six premises (here, $\psi$ is a $\Delta^{b}_{1}(R)$-formula in $S^{1}_{2}(R)$):
\begin{enumerate}
\item $\Gamma, d \geq 1 \land \lnot d \mid n $
\item$\Gamma,\forall i \in [l]. \forall j\in [d].(\exists e \in [n]. \psi(i,j,e)  \lor \forall j' \in [d].  \forall e' \in [n]. \lnot \psi( i,j',e' ))$
\item $\Gamma, \forall (i,j) \in [l]\times[d]. \forall e \neq e^{\prime} \in [n] (\lnot \psi( i,j,e) \lor \lnot \psi( i,j,e^{\prime}))$
\item $\Gamma, \forall (i,j) \neq (i^{\prime},j^{\prime}) \in [l] \times [d]. \forall e \in [n]. (\lnot \psi(i,j,e ) \lor \lnot \psi( i^{\prime},j^{\prime},e ))$
\item{$\Gamma, \forall e \in [n]. \exists (i,j) \in [l] \times [d]. \psi( i,j,e )$}
\end{enumerate}
\end{itemize}
\end{defn}

With Paris-Wilkie translation in mind, a propositional translation of the above first-order calculus can be given as follows: 

\begin{convention}
as for propositional logic, we adopt $\lnot$, binary $\lor,\land$ and unbounded $\bigvee, \bigwedge$ as propositional connectives, $0$ and $1$ as propositional constants, and consider only propositional formulae of negation normal form.
Given a formula $\varphi$, $\overline{\varphi}$ denotes the canonical negation normal form of $\lnot \varphi$. It is called \deff{the complement of $\varphi$}.

For a propositional formula $\varphi$, $|\varphi|$ denotes the size of $\varphi$, say, the number of occurrences of variables and connectives in $\varphi$. 
The precise definition does not matter as long as the conventions are polynomially related. See also Chapter 1 of \cite{proofcomplexity}.

\end{convention}

\begin{defn}
A propositional formula $\varphi$ is $p\Sigma_{i+\frac{1}{2}}(z)$ ($i,z \in \NN$) if and only if it has the following form:
\[
\varphi = \underbrace{\bigvee_{j_{1} \in J_{1}}\bigwedge_{j_{2} \in J_{2}}\cdots}_{\mbox{exactly} \ i\  \mbox{-times}}\psi_{j_{1},\ldots, j_{i}},
\]
 where
 \[ |\varphi| \leq z \ \& \ \ \forall k \in [i].\ J_{k} \neq \emptyset \]
 and each $\psi_{\vec{j}}$ is $\bigvee$- and $\bigwedge$-free and satisfies $|{\psi_{\vec{j}}}| \leq \log z$.
 Note that if $\varphi$ is $p\Sigma_{i+\frac{1}{2}}(z)$, then such $i$ is unique since we distinguish $\bigvee, \bigwedge$ from $\lor, \land$.
 $p\Pi_{i+\frac{1}{2}}(z)$ is defined similarly, switching the roles of $\bigvee$ and $\bigwedge$.
 
 We say $\varphi$ is $s\Sigma_{i+\frac{1}{2}}(z)$ if it is $p\Sigma_{i'+\frac{1}{2}}(z)$ for some $0 \leq i^{\prime} \leq i$, or it is $p\Pi_{i'+\frac{1}{2}}(z)$ for some $0 \leq i^{\prime} < i$.
 Similarly for $s\Pi_{i+\frac{1}{2}}(z)$. 
\end{defn}

 We are particularly interested in the case $z=2^{|n|^{O(1)}}$ for parameters $n$ and $i \leq 2$.

\begin{defn}\label{cedent calculus}
\deff{A cedent} is a finite set of propositional formulae. 
Given a cedent $\Gamma$, its \deff{semantic interpretation} is the propositional formula $\bigvee_{\varphi \in \Gamma}\varphi$. 
Here, if $\Gamma=\emptyset$, we set $\bigvee_{\varphi \in \Gamma}\varphi:=0$.
Under a truth assignment, $\Gamma$ is said to be \deff{true} if and only if its semantic interpretation is true.

We often denote cedents of the form 
\[\Gamma_{1} \cup \ldots \cup \Gamma_{k}\cup \{\varphi_{1},\ldots, \varphi_{m}\}\]
 by 
 \[\Gamma_{1}, \ldots, \Gamma_{k}, \varphi_{1},\ldots, \varphi_{m}.\]
\end{defn}

Given constants $c,d$, a cedent $S$ and a finite vertex-labeled tree $\pi=(\mathcal{T},\mathcal{S})$, $\pi$ is \deff{a $LK^{*}_{1+\frac{1}{2},c}(UCP)$-derivation of $S$ (without redundancy)} if and only if the following hold:
\begin{enumerate}
 \item $height(\mathcal{T}) \leq c$.
 \item For each $v \in \mathcal{T}$, $\mathcal{S}(v)$ is a cedent of cardinality $\leq c$.
 \item $\mathcal{S}(\emptyset) = S$.
 \item For each $v \in \mathcal{T}$, $\mathcal{S}(v)$ is derived from the labels of its children, that is, $(\mathcal{S}(v*i))_{v*i \in \mathcal{T}}$ by applying one of the following derivation rules:
 \begin{itemize}
 \item Initial cedent:
 
 \begin{prooftree}
 \AxiomC{}  \RightLabel{\quad (where $x$ is a constant or a variable)}
  \UnaryInfC{$\Gamma, x, \overline{x}$}

\end{prooftree}
 
 \item $\bigvee$-Rule: 
   \begin{prooftree}
 \AxiomC{$\Gamma, \varphi_{i_{0}}$} \RightLabel{\quad (where $\bigvee_{i=1}^{I}\varphi_{i} \in \Gamma$, $1\leq i_{0} \leq I$, $\varphi_{i_{0}} \not \in \Gamma$)}
  \UnaryInfC{$\Gamma$}
 \end{prooftree}
 
\item $\lor$-Rule: 
   \begin{prooftree}
 \AxiomC{$\Gamma, \varphi_{i_{0}}$} \RightLabel{\quad (where $\varphi_{1} \lor \varphi_{2} \in \Gamma$, $i_{0}=1$ or $i_{0}=2$, $\varphi_{i_{0}} \not \in \Gamma$)}
  \UnaryInfC{$\Gamma$}
 \end{prooftree}

 \item $\bigwedge$-Rule:
    \begin{prooftree}
 \AxiomC{$\Gamma, \varphi_{1}$}
 \AxiomC{$\Gamma, \varphi_{2}$}
 \AxiomC{$\cdots$}
 \AxiomC{$\Gamma, \varphi_{I}$}
   
  \QuaternaryInfC{$\Gamma$}
 \end{prooftree}
 where $\bigwedge_{i=1}^{I}\varphi_{i} \in \Gamma$, and $\varphi_{i} \not \in \Gamma$ for each $i \in [I]$.
 
 \item $\land$-Rule:
    \begin{prooftree}
 \AxiomC{$\Gamma, \varphi_{1}$}
 \AxiomC{$\Gamma, \varphi_{2}$}
  \BinaryInfC{$\Gamma$}
 \end{prooftree}
 where $\varphi_{1} \land \varphi_{2} \in \Gamma$, and $\varphi_{i} \not \in \Gamma$ for each $i=1,2$.

 \item Trivial Cut:
  \begin{prooftree}
 \AxiomC{$\Gamma, 0$} \RightLabel{\quad (where $0 \not \in \Gamma$)}
  \UnaryInfC{$\Gamma$}
 \end{prooftree}

 \item $p\Sigma_{1+\frac{1}{2}}$-Induction:
     \begin{prooftree}
 \AxiomC{$\Gamma, \varphi_{1}$}
 \AxiomC{$\Gamma, \overline{\varphi_{1}},\varphi_{2}$}
 \AxiomC{$\cdots$}
 \AxiomC{$\Gamma, \overline{\varphi_{I-1}},\varphi_{I}$}
  \AxiomC{$\Gamma, \overline{\varphi_{I}}$}
  \QuinaryInfC{$\Gamma$}
 \end{prooftree}
 where each $\varphi_{i}$ is $p\Sigma_{1+\frac{1}{2}}(|{\pi}|)$, and $\varphi_{i}, \overline{\varphi}_{i}  \not \in \Gamma$ for $i \in [I]$.
 Here, we define the size $|{\pi}|$ of the proof $\pi$ as 
 \[|{\pi}| := \sum_{v \in \mathcal{T}} \sum_{\varphi \in \mathcal{S}(v)} |{\varphi}|.\]
Note that, when $I=1$, $p\Sigma_{1+\frac{1}{2}}$-Induction is a usual cut-rule for $p\Sigma_{1+\frac{1}{2}}(|{\pi}|)$-formulae.

\item\label{UCP-axiom} (Propositional) $UCP$-axiom:
deriving $\Gamma$ from the following four kinds of premises, where $k \not \equiv 0 \pmod d$, $d \geq 1$, $l \geq 1$, and each $\theta_{i,j,e}$ is $p\Sigma_{0+\frac{1}{2}}$-formulae:
\begin{enumerate}
\item $\Gamma, \bigvee_{e \in [k]} \theta_{i,j,e} \lor \bigwedge_{e' \in [k], j' \in [d]} \lnot \theta_{i,j',e'}$ \quad ($(i,j) \in [l] \times [d]$).
\item $\Gamma, \lnot \theta_{i,j,e} \lor \lnot \theta_{i,j,e^{\prime}}$ \quad ($(i,j) \in [l] \times [d], e \neq e^{\prime} \in [k]$).
%\item $\Gamma,\forall i \in [l]. \forall j\in [d].\forall j' \in [d]. (\exists e \in [n]. R(i,j,e)  \lor \forall e' \in [n]. \lnot R( i,j',e' ))$
%\item $\Gamma, \forall (i,j) \in [l]\times[d]. \forall e \neq e^{\prime} \in [n] (\lnot R( i,j,e) \lor \lnot R( i,j,e^{\prime}))$
\item $\Gamma, \lnot \theta_{i,j,e} \lor \lnot \theta_{i^{\prime},j^{\prime},e}$ \quad ($(i,j) \neq (i^{\prime},j^{\prime}) \in [l] \times [d]$, $e \in [k]$).
%\item $\Gamma, \forall (i,j) \neq (i^{\prime},j^{\prime}) \in [l] \times [d]. \forall e \in [n]. (\lnot R(i,j,e ) \lor \lnot R( i^{\prime},j^{\prime},e ))$
\item $\Gamma, \bigvee_{(i,j) \in [l] \times [d]} \theta_{i,j,e}$ \quad ($e \in [k]$).
%\item{$\Gamma, \forall e \in [n]. \exists (i,j) \in [l] \times [d]. R( i,j,e )$}
\end{enumerate}

\end{itemize}
\end{enumerate}

It is a routine to check the following with the idea of Paris-Wilkie translation:
\begin{prop}
If $T^{1}_{2}(R)+UCP^{l,d}_{k}(\Delta^{b}_{1}(R)) \vdash injPHP^{n+1}_{n}(R)$, then $injPHP^{n+1}_{n}$, expressed as the cedent
\begin{align*}
\left\{ \bigvee_{p \in [n+1]}\bigwedge_{h \in [n]} \lnot r_{ph}, \bigvee_{\substack{p \neq p^{\prime} \in [n+1],\\ h \in [n]}} (r_{ph} \land r_{p^{\prime}h}),
 %\bigvee_{h \in [n]}\bigwedge_{p \in [n+1]} \lnot r_{ph}, 
 \bigvee_{\substack{h \neq h^{\prime} \in [n],\\ p \in [n+1]}} (r_{ph} \land r_{ph^{\prime}}) \right\},
\end{align*}
 has $2^{|n|^{O(1)}}$-sized $LK^{*}_{1+\frac{1}{2},O(1)}(UCP)$-derivation.
\end{prop}

\begin{rmk}[The subformula-property of $LK^{*}_{1+\frac{1}{2},c}(UCP)$-derivations]
Note that $LK^{*}_{1+\frac{1}{2},c}(UCP)$-derivations of $injPHP^{n+1}_{n}$ consist only of:
\begin{enumerate}
 \item subformulae of $injPHP^{n+1}_{n}$.
 \item formulae $\bigvee_{e \in [k]} \theta_{i,j,e} \lor \bigwedge_{e' \in [k]} \lnot \theta_{i,j',e'}$ from $UCP$-axiom.
 \item $s\Sigma_{1+\frac{1}{2}}$- or $s\Pi_{1+\frac{1}{2}}$-formulae. 
 \end{enumerate}
\end{rmk}

In the next sections, we deny the existence of such short $LK^{*}_{1+\frac{1}{2},c}(UCP)$-derivations of $injPHP^{n+1}_{n}$, which results in the independence 
\[T^{1}_{2}(R)+UCP^{l,d}_{k}(\Delta^{b}_{1}(R)) \not\vdash injPHP^{n+1}_{n}(R).\]

\section{$o(n)$-evaluation for depth $1+\frac{1}{2}$}\label{evaluation}
This section is based on the notions defined in \cite{UCPvsinjPHPrevised} \S 3.
We first show the following theorem (cf. Definition 48 in \cite{UCPvsinjPHPrevised}):
\begin{thm}\label{k-evaluationexists}
There exists a positive constant $\epsilon$ such that, for an arbitrary sufficiently large natural number $n$ and a subformula-closed family $\Gamma$ of $\{r_{ij}\}_{i \in [n+1], j \in [n]}$-formulae within depth $1+\frac{1}{2}$ (that is, $s\Sigma_{1+\frac{1}{2}}$- or $s\Pi_{1+\frac{1}{2}}$-formulae), if $\#\Gamma < 2^{n^{\epsilon}}$, then there exists a partial \textbf{bijection} $\rho \in \mathcal{M}^{n+1}_{n}$ of size $n-n^{\epsilon}$ such that $\Gamma^{\rho}$ has an $o(n^{\epsilon})$-evaluation with $injPHP$-trees.
\end{thm}

\begin{defn}
For $n \geq s$, let $\mathcal{B}^{n}_{s}$ be the set of all the partial \textbf{bijections} between $[n+1]$ and $[n]$ with size $s$. 
\end{defn}

Then we obtain:

\begin{cor}\label{k-evaluationforproofs}
 Suppose $injPHP^{n+1}_{n}$ has $2^{|n|^{O(1)}}$-sized $LK^{*}_{1+\frac{1}{2},c}(UCP)$-derivation.
 Let $\pi_{n}$ ($n \geq 1$) be such derivations.
 Set $\Gamma_{n}$ as the set of all the formulae and their subformulae in $\pi_{n}$.
  Then there exists a positive constant $\epsilon$ such that, for sufficiently large $n$, there exists $\rho \in \mathcal{B}^{n}_{n-n^{\epsilon}}$ such that $\Gamma^{\rho}$ has an $o(n^{\epsilon})$-evaluation with $injPHP$-trees.
\end{cor}

For convenience, we introduce the following notation:
\begin{defn}
Given two $\{0,1\}$-labeled trees $S_{1}=(T_{1},L_{1})$ and $S_{2}=(T_{2},L_{2})$ with $height(T_{1})+height(T_{2}) \leq n$ (cf. Definition 37 in \cite{UCPvsinjPHPrevised}), 
\[S_{1} \lor S_{2} := (T_{1}*T_{2}, L),\]
where $L(bb') = L_{1}(b) \lor L_{2}(b')$ for $b \in br(T_{1})$ and $b' \in br(T_{2})$ with $b||b'$. (See also Definition 46 in \cite{UCPvsinjPHPrevised} for the definition of $*$.)

\end{defn}

\begin{rmk}
Note that $\lor$ is associative as long as the sum of all the heights of the trees in concern is still no more than $n$.
Therefore, we omit parentheses for $\lor$ of constantly many $\{0,1\}$-labeled trees of height $o(n)$.
\end{rmk}

\begin{proof}[Proof of Corollary \ref{k-evaluationforproofs}]
Let $\widetilde{\Gamma}_{n}$ be the formulae in $\Gamma_{n}$ within depth $1+\frac{1}{2}$.
Then, by Theorem \ref{k-evaluationexists}, we obtain $\rho \in \mathcal{B}^{n}_{n-n^{\epsilon}}$ such that $\widetilde{\Gamma}^{\rho}$ has an $o(n^{\epsilon})$-evaluation $T_{n}$ with $injPHP$-trees for sufficiently large $n$.
It suffices to extend $T_{n}$ on $\widetilde{\Gamma}_{n}$ to $\Gamma_{n}$.

For readability, we suppress superscripts $\rho$ and replace $n^{\epsilon}$ with $n$ below, as if the restricted formulae were the original.

Due to subformula-property of $\pi_{n}$, $\Gamma_{n}\setminus \widetilde{\Gamma}_{n}$ consists of:
\begin{itemize}
 \item the formula $\bigvee_{p \in [n+1]}\bigwedge_{h \in [n]} \lnot r_{ph}$ in $injPHP^{n+1}_{n}$.
 \item subformulae $\bigvee_{e \in [k]} \theta_{i,j,e} \lor \bigwedge_{e' \in [k]} \lnot \theta_{i,j',e'}$ (each $\theta_{i,j,e}$ is $p\Sigma_{0+\frac{1}{2}}$).
\end{itemize}

For the formula of the first kind, assume $T_{n}$ assigns an $injPHP$-tree $S_{p}$ for each $\bigwedge_{h \in [n]} \lnot r_{ph}$ ($p \in [n+1]$).
Fix $p \in [n+1]$.
By definition of $injPHP$-trees, each $b \in br_{1}(S_{p})$ is incompatible with any size-1 matching of the form $\langle p,h \rangle$ ($h \in [n]$).
On the other hand, since the height of $S_{p}$ is $o(n)$, for each $b \in br(S_{p})$, there exists $h \in [n]$ such that $b || \langle p,h \rangle$.
Therefore, we have $br_{0}(S_{p})=br(S_{p})$.
Since $p$ is arbitrary, it suffices to assign the tree of height $0$ and with the label $0$ to $\bigvee_{p \in [n+1]}\bigwedge_{h \in [n]} \lnot r_{ph}$.

Now, we consider the formulae of the second kind, that is, we describe an appropriate assignment for the formula of the form
\begin{align}
\bigvee_{e \in [k]} \theta_{i,j,e} \lor \bigwedge_{e' \in [k]} \lnot \theta_{i,j',e'} \label{emptyordset}
\end{align}
(each $\theta_{i,j,e}$ is $p\Sigma_{0+\frac{1}{2}}$).
$T_{n}$ is already defined for each disjunct $\bigvee_{e \in [k]} \theta_{i,j,e}$ and $\bigwedge_{e' \in [k]} \lnot \theta_{i,j',e'}$.
Let $U$ and $U'$ be assigned trees respectively.
Then it suffices to assign $U \lor U'$ to the whole formula (\ref{emptyordset}).

\end{proof}

Our proof of Theorem \ref{k-evaluationexists} is analogous to the well-established proof of the existence of $k$-evaluation for Ajtai's theorem in \cite{remake} and \cite{remake2} (for a neat and modern exposition with historical backgrounds, see \cite{proofcomplexity} Chapter 15). 
A main difference here is that we have to control the number of singletons in the partial injections corresponding to the branches in our shallow $injPHP$-trees, which is accomplished by first querying only pigeons to decide formulae of depth $0+\frac{1}{2}$ and then by allowing hole-queries for formulae of depth $1+\frac{1}{2}$.

\begin{proof}[Proof of Theorem \ref{k-evaluationexists}]
As for each $p\Sigma_{0+\frac{1}{2}}$-formula $\theta \in \Gamma$, since all the atoms in $\theta$ are among $\{r_{ij}\}_{i \in [n+1], j \in [n]}$, we can decide the truth value of $\theta$ by partial assignments induced by partial \textbf{bijections} covering all the pigeons appearing (as indices of the variables) in $\theta$.
Since $|{\theta}| \leq (\log(n))^{O(1)}$, it suffices to assign an $injPHP$-trees querying all the pigeons in $\theta$ one by one, labeling each branch by the corresponding truth value of $\theta$.

 Now, we consider an assignment for $p\Sigma_{1+\frac{1}{2}}$-formulae in $\Gamma$.
 We need a partial bijection which works for at most $\#\Gamma$-many $p\Sigma_{1+\frac{1}{2}}$-formulae simultaneously, and we begin with counting partial bijections working for a fixed formula.
 Without loss of generality, we consider a formula $\varphi$ of the form $\bigvee_{i \in [L]} \psi_{i}$, where each $\psi_{i}$ is $p\Sigma_{0+\frac{1}{2}}$.
 Let $\epsilon$ be a positive parameter, which will be optimized later.
 Take $\boldsymbol{\rho} \in \mathcal{B}^{n}_{n-n^{\epsilon}}$ uniformly randomly.
 We consider the canonical \textbf{$injPHP$-tree} (therefore each branch corresponds to partial \textbf{injections} in general) for 
 \[\mathcal{F}^{\boldsymbol{\rho}}:=\bigcup_{i \in [L]} br_{1}((T_{n}(\psi_{i}))^{\boldsymbol{\rho}}),\]
  described in Example 49 in \cite{UCPvsinjPHPrevised}.
 We denote the tree by $S_{\boldsymbol{\rho}}$ below.
 Our aim is to bound the number of $\boldsymbol{\rho}$ for which $height(S_{\boldsymbol{\rho}}) > s:=n^{\epsilon/2}$.
 
 We code $\boldsymbol{\rho}$ by a triple $(\chi, \vec{\beta}, \vec{\gamma})$, where
 \begin{itemize}
  \item $\chi$ is a partial bijection of size $n-n^{\epsilon}+r$ ($(s/2q)-1 \leq r \leq s$, and $q:=log_{2}(|\pi_{n}|)$. Note that $|{\theta}| \leq q$ holds for any $\theta$ occurring in $\pi_{n}$ with depth $0+\frac{1}{2}$).
  \item $\vec{\beta}=(\beta_{1},\ldots,\beta_{l}) \in (\{0,1\}^{q})^{l}$ ($(s/2q)-q \leq l \leq s/q$).
  We may assume $\vec{\beta} \in \{0,1\}^{s}$.
  \item $\vec{\gamma}=(\gamma_{1}, \ldots, \gamma_{l}) \in ([n^{\epsilon}]^{2q})^{l}$.
  We may assume $\vec{\gamma} \in [n^{\epsilon}]^{2s}$.
 \end{itemize}
 
 Each item is concretely defined as follows: take the leftmost path in $S_{\boldsymbol{\rho}}$ among those with length $\geq s$.
 Let $\tau_{1}, \ldots, \tau_{N}$ be the elements of $\mathcal{F}$ which was considered when the path was constructed.
 Let $l+1$ be the first index when the path got longer than $s$. 
 Since each $\tau_{i}$ is of size at most $q$, we have 
 \[lq \leq s \leq 2(l+1)q, \ \mbox{that is,}\ s/(2q)-2q\leq l \leq s/q.\]
 
 Let $\nu_{i}$ ($i \in [l]$) be the resulting path considering $\tau_{i}$.
 Note that $\tau_{i}$ is compatible with $\boldsymbol{\rho}\nu_{1}\cdots\nu_{i-1}$, and $\nu_{i}$ covers all the pigeons and holes in $\tau_{i}^{\boldsymbol{\rho}\nu_{1}\cdots\nu_{i-1}}$.

 \begin{enumerate}
  \item Set $\chi:=\boldsymbol{\rho}\tau_{1}\cdots\tau_{l}$. Note that $\tau_{j+1} || \boldsymbol{\rho}\tau_{1}\cdots\tau_{j}$ immediately follows by induction.
  \item Let $\beta_{i}\in \{0,1\}^{q}$ be the indicator vector of 
  \begin{align}\label{map}
  \tau_{i}^{\boldsymbol{\rho}\nu_{1}\cdots\nu_{i-1}}=\tau_{i}^{\boldsymbol{\rho}\nu_{1}\cdots\nu_{i-1}\tau_{i+1}^{\boldsymbol{\rho}\nu_{1}\cdots\nu_{i}}\cdots\tau_{l}^{\boldsymbol{\rho}\nu_{1}\cdots\nu_{l-1}}}
  \end{align}
   in $\tau_{i}$, indicating which disjunct in $\tau_{i}$ survives after the restriction by $\boldsymbol{\rho}\nu_{1}\cdots\nu_{i-1}$.
  \item\label{gamma} Let $\gamma_{i}\in [n^{\epsilon}]^{2q}$ be a vector indicating which holes (resp. pigeons) are mapped to the pigeons (resp. holes) in the previous partial bijection (\ref{map}) by $\nu_{i}$.
  More precisely, $\gamma_{i}$ is a code of a partial injection from $([n+1]\setminus \dom(\xi))$ to $([n] \setminus \ran(\xi))$ of size at most $2q$, where 
  \[\xi:=\boldsymbol{\rho}\nu_{1}\cdots\nu_{i-1}\tau_{i+1}^{\boldsymbol{\rho}\nu_{1}\cdots\nu_{i}}\cdots\tau_{l}^{\boldsymbol{\rho}\nu_{1}\cdots\nu_{l-1}}.\]
 \end{enumerate}
 
 Then $\boldsymbol{\rho} \mapsto (\chi, \vec{\beta}, \vec{\gamma})$ is injective.
 Indeed, $\boldsymbol{\rho}$ is recovered from $(\chi, \vec{\beta}, \vec{\gamma})$ as follows:
 recall that $\mathcal{F}^{\boldsymbol{\rho}}$ was linearly ordered in the first place in the construction of $S_{\boldsymbol{\rho}}$.
 Towards the recovery of $\boldsymbol{\rho}$, we find $\tau_{1}^{\boldsymbol{\rho}}, \nu_{1}, \tau_{2}^{\boldsymbol{\rho}\nu_{1}}, \ldots, \tau_{l}^{\boldsymbol{\rho}\nu_{1}\cdots\nu_{l-1}}$.
 Then, subtracting $\tau_{1}^{\boldsymbol{\rho}}\cdots\tau_{l}^{\boldsymbol{\rho}\nu_{1}\cdots\nu_{l-1}}$ from $\chi$, we get $\boldsymbol{\rho}$.
 
 Assume $\tau_{1}^{\boldsymbol{\rho}}, \nu_{1}, \ldots, \tau_{i-1}^{\boldsymbol{\rho}\nu_{1}\cdots\nu_{i-2}}, \nu_{i-1}$ are already recovered ($i \in [l]$).
 Then, compared with $\chi$, we obtain 
 \[\sigma:=\boldsymbol{\rho}\nu_{1}\cdots\nu_{i-1}\tau_{i}^{\boldsymbol{\rho}\nu_{1}\cdots\nu_{i-1}}\cdots\tau_{l}^{\boldsymbol{\rho}\nu_{1}\cdots\nu_{l-1}}.\]
 Thus we can find $\tau_{i}$ as the first element in $\mathcal{F}^{\boldsymbol{\rho}}$ compatible with $\sigma$.
 Now, with the data of $\beta_{i}$, we obtain $\tau_{i}^{\boldsymbol{\rho}\nu_{1}\cdots\nu_{i-1}}$, and hence 
 \[\sigma \setminus \tau_{i}^{\boldsymbol{\rho}\nu_{1}\cdots\nu_{i-1}} = \boldsymbol{\rho}\nu_{1}\cdots\nu_{i-1}\tau_{i+1}^{\boldsymbol{\rho}\nu_{1}\cdots\nu_{i}}\cdots\tau_{l}^{\boldsymbol{\rho}\nu_{1}\cdots\nu_{l-1}},\]
 which is $\xi$ in the item (\ref{gamma}) above.
 Therefore, with the data of $\gamma_{i}$, we can recover $\nu_{i}$.

 Lastly, we count the number of the possible triples $(\chi, \vec{\beta}, \vec{\gamma})$ as above.
 We consider the ratio of the number to the number of possible $\boldsymbol{\rho}$:
 \begin{align*}
 &\frac{\# \left(\bigcup_{(s/2q)-1 \leq r \leq s} \mathcal{B}^{n}_{n-n^{\epsilon}+r} \right)\times \# \{0,1\}^{s} \times \#[n^{\epsilon}]^{2s}}{\#\mathcal{B}^{n}_{n-n^{\epsilon}}} \\
 \leq  &2^{s}n^{2\epsilon s}\sum_{(s/2q)-1 \leq r \leq s}\frac{ \# \mathcal{B}^{n}_{n-n^{\epsilon}+r}}{\#\mathcal{B}^{n}_{n-n^{\epsilon}}} \\
 = &2^{s}n^{2\epsilon s}\sum_{(s/2q)-1 \leq r \leq s} \frac{\binom{n^{\epsilon}}{r}(n^{\epsilon}+1)^{(r)}}{\binom{n-n^{\epsilon}+r}{r}}\\
 = &2^{s}n^{2\epsilon s}\sum_{(s/2q)-1 \leq r \leq s} \frac{(n^{\epsilon})^{(r)}(n^{\epsilon}+1)^{(r)}}{(n-n^{\epsilon}+r)^{(r)}}\\
 \leq &2^{s}n^{2\epsilon s}\sum_{(s/2q)-1 \leq r \leq s} \frac{n^{\epsilon r} (n^{\epsilon}+1)^{r}}{(n-n^{\epsilon})^{r}}\\
 = &2^{s}n^{2\epsilon s}\sum_{(s/2q)-1 \leq r \leq s} \left(\frac{n^{\epsilon} (n^{\epsilon}+1)}{n-n^{\epsilon}}\right)^{r}\\
 = & 2^{s}n^{2\epsilon s} \frac{\left(\frac{n^{\epsilon} (n^{\epsilon}+1)}{n-n^{\epsilon}}\right)^{(s/2q)-1}}{1-\frac{n^{\epsilon} (n^{\epsilon}+1)}{n-n^{\epsilon}}}\\
 =& O\left(2^{s}n^{2\epsilon s}\left(\frac{n^{\epsilon} (n^{\epsilon}+1)}{n-n^{\epsilon}}\right)^{(s/2q)-1}\right)
 \end{align*}
  
 Now, recalling $s=n^{\epsilon/2}$ and $q=\log_{2}(n)^{O(1)}$, the rightmost side is further bounded by 
 \[O(n^{(5\epsilon-1)(s/2q-1)}) \leq 2^{-\Omega(n^{\epsilon/3})}\]
 if $\epsilon > 0$ is set sufficiently small.
 
 $\varphi$ was arbitrary. Since there are at most $2^{|n|^{O(1)}}$-many of them, and 
 \[2^{|n|^{O(1)}} \times 2^{-\Omega(n^{\epsilon/3})} = o(1),\]
  we obtain the desired $\boldsymbol{\rho}$ by union bound.
 
\end{proof}

\begin{rmk}
Note that the construction of a restriction and an $o(n^{\epsilon})$-evaluation above cannot be extended to depth $2+\frac{1}{2}$-case straightforwardly.
It is totally nontrivial to bound the number of triples $(\chi,\vec{\beta},\vec{\gamma})$ in this case since $\chi$ can be a partial injection in general: each $\tau_{i}$ can include singletons. 
Although we can show that at least a half of $\tau_{i}$ are 2-sets (of pigeons and holes), we cannot exclude the possibility that $\tau_{i}^{\nu_{1}\cdots\nu_{i-1}}$ consists, say, only of singletons. 

Indeed, if we consider the formula $\bigvee_{h \in [n]}\bigwedge_{p \in [n+1]} \lnot r_{ph}$, we face the canonical $injPHP$-tree for $\{\{\langle h \rangle\} \mid h \in [n]\}$.
After any restriction with a partial bijection in $\mathcal{B}^{n}_{n-n^{\epsilon}}$, there still remains $n^{\epsilon}$-many singletons, and it is straightforward to show that the canonical $injPHP$-tree will not be shallow (the height grows to $n^{\epsilon}$) in this case.
\end{rmk}

\section{The Main Result}\label{The Main Result}
In this section, we show:
\begin{thm}
The cedents $injPHP^{n+1}_{n}$ ($n \geq 1$) do not have $2^{|n|^{O(1)}}$-sized $LK^{*}_{1+\frac{1}{2},O(1)}(UCP)$-derivations.
\end{thm}

The proof is analogous to the proof of Theorem 53 in \cite{UCPvsinjPHPrevised}.
The main difference is that we have $o(n)$-evaluation for only formulae within depth $1+\frac{1}{2}$ this time while the instances of $UCP^{l,d}_{k}$ in our concern have depth $2+\frac{1}{2}$.
We deal with it by expanding the outermost connective in the formulation of $LK_{1+\frac{1}{2}}^{*}(UCP)$ (see Definition \ref{cedent calculus}).

\begin{proof}
Suppose $injPHP^{n+1}_{n}$ ($n \geq 1$) had $2^{|n|^{O(1)}}$-sized $LK^{*}_{1+\frac{1}{2},O(1)}(UCP)$-derivations $\pi_{n}$.
By Corollary \ref{k-evaluationforproofs}, for an appropriate $\epsilon > 0$ and for sufficiently large $n$, there exists $\rho_{n} \in \mathcal{B}^{n}_{n-n^{\epsilon}}$ and an $o(n^{\epsilon})$-evaluation $T_{n}$ for the subformula-closure of the set of $\varphi^{\rho}$, where $\varphi$ ranges over all the formulae occurring in $\pi_{n}$.
Below, for readability, we suppress superscripts $\rho$ of $\varphi^{\rho}$ and replace $n^{\epsilon}$ with $n$, as if the restricted formulae were the original. 

Fix a sufficiently large $n$ as above.
We denote $T_{n}$ by $T$, and $T_{\varphi}$ denotes the $\{0,1\}$-labeled $injPHP$-tree assigned to the formula $\varphi$.
For each sequent $\Gamma$ in $\pi$, we write $T \models \Gamma$ if 
\[U:=T_{\varphi_{1}}\lor \cdots \lor T_{\varphi_{l}} \quad \mbox{ (where $\Gamma = \{\varphi_{1},\ldots, \varphi_{l}\}$)}\]
 satisfies $br_{1}(U)=br(U)$.
Note that $l=O(1)$ and $height(T_{i})=o(n)$ ($i \in [l]$), and therefore $U$ is well-defined as a labeled $injPHP$-tree.

It is straightforward to see that the semantics $T\models \varphi$ is sound for all the derivation rules except for $UCP$-axiom (cf. Claim 54 in \cite{UCPvsinjPHPrevised}).
On the other hand, $T \not\models injPHP^{n+1}_{n}$ (cf. Claim 55 in \cite{UCPvsinjPHPrevised}).

Therefore, there exists an instance of $UCP$-axiom as Definition \ref{cedent calculus} (\ref{UCP-axiom}) in $\pi_{n}$ such that:
 \begin{enumerate}
\item $T \models \Gamma \cup\{ \bigvee_{e \in [k]} \theta_{i,j,e} \lor \bigwedge_{e' \in [k], j' \in [d]} \lnot \theta_{i,j',e'} \}$ \quad ($(i,j) \in [l] \times [d]$).
\item $T \models \Gamma \cup\{ \lnot \theta_{i,j,e} \lor \lnot \theta_{i,j,e^{\prime}}\}$ \quad ($(i,j) \in [l] \times [d], e \neq e^{\prime} \in [k]$).
%\item $\Gamma,\forall i \in [l]. \forall j\in [d].\forall j' \in [d]. (\exists e \in [n]. R(i,j,e)  \lor \forall e' \in [n]. \lnot R( i,j',e' ))$
%\item $\Gamma, \forall (i,j) \in [l]\times[d]. \forall e \neq e^{\prime} \in [n] (\lnot R( i,j,e) \lor \lnot R( i,j,e^{\prime}))$
\item $T \models \Gamma\cup \{\lnot \theta_{i,j,e} \lor \lnot \theta_{i^{\prime},j^{\prime},e}\}$ \quad ($(i,j) \neq (i^{\prime},j^{\prime}) \in [l] \times [d]$, $e \in [k]$).
%\item $\Gamma, \forall (i,j) \neq (i^{\prime},j^{\prime}) \in [l] \times [d]. \forall e \in [n]. (\lnot R(i,j,e ) \lor \lnot R( i^{\prime},j^{\prime},e ))$
\item $T \models \Gamma\cup\{ \bigvee_{(i,j) \in [l] \times [d]} \theta_{i,j,e}\}$ \quad ($e \in [k]$).
%\item{$\Gamma, \forall e \in [n]. \exists (i,j) \in [l] \times [d]. R( i,j,e )$}
\item $T \not \models \Gamma$.
\end{enumerate}

As for the last item, with further restriction by a partial injection of size $\log(n)^{O(1)}$, we may assume that $br_{0}(T_{\varphi})=br(T_{\varphi})$ for every $\varphi \in \Gamma$.
Therefore, we have
 \begin{enumerate}
\item\label{emptyord-set} $T \models  \bigvee_{e \in [k]} \theta_{i,j,e} \lor \bigwedge_{e' \in [k], j' \in [d]} \lnot \theta_{i,j',e'} $ \quad ($(i,j) \in [l] \times [d]$).
\item\label{incompatible1} $T \models  \lnot \theta_{i,j,e} \lor \lnot \theta_{i,j,e^{\prime}}$ \quad ($(i,j) \in [l] \times [d], e \neq e^{\prime} \in [k]$).
%\item $\Gamma,\forall i \in [l]. \forall j\in [d].\forall j' \in [d]. (\exists e \in [n]. R(i,j,e)  \lor \forall e' \in [n]. \lnot R( i,j',e' ))$
%\item $\Gamma, \forall (i,j) \in [l]\times[d]. \forall e \neq e^{\prime} \in [n] (\lnot R( i,j,e) \lor \lnot R( i,j,e^{\prime}))$
\item\label{incompatible2} $T \models \lnot \theta_{i,j,e} \lor \lnot \theta_{i^{\prime},j^{\prime},e}$ \quad ($(i,j) \neq (i^{\prime},j^{\prime}) \in [l] \times [d]$, $e \in [k]$).
%\item $\Gamma, \forall (i,j) \neq (i^{\prime},j^{\prime}) \in [l] \times [d]. \forall e \in [n]. (\lnot R(i,j,e ) \lor \lnot R( i^{\prime},j^{\prime},e ))$
\item\label{totality} $T \models  \bigvee_{(i,j) \in [l] \times [d]} \theta_{i,j,e}$ \quad ($e \in [k]$).

\end{enumerate}

We observe the following:
\begin{enumerate}
\renewcommand{\labelenumi}{\roman{enumi}}
 \item\label{ijtot} By (\ref{totality}), for each $e \in [k]$, every $b \in br(S_{e})$ has $(i,j) \in [l] \times [d]$ and $b' \in br_{1}(S_{i,j,e})$ such that $b' \subseteq b$.
 \item\label{einj} By (\ref{incompatible1}) above, for each $e \in [k]$ and $(i,j) \neq (i^{\prime}, j^{\prime}) \in [l] \times [d]$, every pair of branches $b \in br_{1}(S_{i,j,e})$ and $b^{\prime} \in br_{1}(S_{i^{\prime},j^{\prime},e})$ satisfies $b \perp b^{\prime}$.
 \item\label{ijinj} By (\ref{incompatible2}) above, for each $e \neq e^{\prime} \in [k]$ and $(i,j) \in [l] \times [d]$, each $b \in br_{1}(S_{i,j,e})$ and $b^{\prime} \in br_{1}(S_{i,j,e^{\prime}})$ satisfies $b \perp b^{\prime}$.
 %\item\label{iquasitot} For each $i \in [l]$, $T \models (\lnot\bigvee_{j \in [d]} \lnot \bigvee_{e \in [k]}\psi_{i,j,e}) \lor (\lnot \bigvee_{j \in [d]} \bigvee_{e \in [k]}\psi_{i,j,e})$, that is, every $b \in br(U_{i})$ is an extension of some $b^{\prime} \in br_{1}(P_{i}^{c})$ or some $b^{\prime} \in br_{1}(N_{i}^{c})$. In the former case, $b$ is incompatible with every $b^{\prime\prime}\in \bigcup_{j}br_{0}(S_{i,j})$. In the latter case, $b$ is incompatible with every $b^{\prime\prime} \in \bigcup_{j,e}br_{1}(S_{i,j,e})$. Therefore, the two cases are mutually disjoint. Indeed, if $b$ satisfies the both cases, then take $b^{\prime} \in br(S_{i,j})$ such that $b || b^{\prime}$ (which exists since $\#b$ and $height(S_{i,j})$ are both $o(n)$). It follows that $b^{\prime} \in br_{1}(S_{i,j})$, and therefore $b^{\prime}$ is an extension of some $b^{\prime \prime } \in \bigcup_{e}br_{1}(S_{i,j,e})$, which contradicts the observation of the latter case. 
\end{enumerate}

With the previous observations, we construct labeled $injPHP$-trees $(X_{i,j})_{(i,j) \in [l] \times [d]}$ and $(Y_{e})_{e \in [k]}$ as follows：
\begin{itemize}
 \item Let 
 \begin{align*}
 S_{e} &:=T_{\bigvee_{(i,j) \in [l] \times [d]} \theta_{i,j,e}} \quad (e \in [k]), \\
 U_{i,j} &:=N_{i} * \sum_{b \in br_{0}(N_{i})}S_{i,j}^{b} \quad ((i,j) \in [l] \times [d]),
 \end{align*}
 where $N_{i}:=T_{\bigwedge_{e' \in [k], j' \in [d]} \lnot \theta_{i,j',e'}}$, $S_{i,j}:= T_{\bigvee_{e \in [k]} \theta_{i,j,e}}$.
  By (\ref{emptyord-set}), if $b \in br_{0}(N_{i})$, then for any $(b')^{b} \in br_{1}(S_{i,j}^{b})$, $b' \in br_{1}(S_{i,j})$, that is, together with the observation (ii), there uniquely exists $e_{b} \in [k]$ such that $b$ extends some $b'' \in br_{1}(T_{\theta_{i,j,e_{b}}})$.
  
 \item We define $Y_{e}$ for fixed $e$ first.  
 By observations (i)(iii), each $b \in br(S_{e})$ has a unique $(i_{b},j_{b}) \in [l] \times [d]$ such that $b$ is an extension of some $b^{\prime} \in br_{1}(S_{i_{b},j_{b},e})$. 
 Consider the tree 
 \begin{align*}
 S_{e}*\sum_{b \in br(S_{e})}(U_{i_{b},j_{b}})^{b}.
 \end{align*}
 (here, we have concatenated the trees, ignoring their labels).
 
 Label each branch extending $b \in br(S_{e})$ with $\langle i_{b}, j_{b},e\rangle$. 
 Let $Y_{e}$ be the resulting labeled $injPHP$-tree. 
 Note that $height(Y_{e})$ is still $o(n)$.
 
 \item Next, we define $X_{i,j}$ for fixed $(i,j) \in [l]\times[d]$. 
 Let $B \subseteq br(U_{i,j})$ be the set of all $\tau \tau' \in br(U_{i,j})$ such that $\tau \in br_{0}(N_{i}), \tau' \in br_{1}(S_{i,j})$. 
 Let 
 \[X_{i,j}:= U_{i,j}*\sum_{b \in B} S_{e_{b}}^{b}.\]
 
 Label each branch of $X_{i}$ extending $b \in B$ by $\langle i, j, e_{b} \rangle$ and the remaining branched by $\perp$.
 \end{itemize}
 
 We see that $(X_{i,j})_{i \in [l],j \in [d]}$ and $(Y_{e})_{e \in [k]}$ satisfy the following:
\begin{itemize}
 \item For each $i$，$br_{\perp}(X_{i,1})= \cdots = br_{\perp}(X_{i,d}) = br_{1}(N_{i})$.
  \item For each $i,j,e$, $br_{\langle i,j,e \rangle}(X_{i,j})= br_{\langle i,j,e \rangle}(Y_{e})$ (as sets of partial injections). 
\end{itemize}

Now, we are ready to repeat the construction of a low-degree NS-refutation of $injPHP^{n+1}_{n}$ over the ring $\ZZ_{d}$ given in the proof of Theorem 53 in \cite{UCPvsinjPHPrevised}, which contradicts Razborov's linear degree lower bound for $injPHP^{n+1}_{n}$ (\cite{Razborov}). (See also \cite{UCPvsinjPHPrevised} Corollary 25.)

\end{proof}

\section{A model theoretic approach} \label{A model theoretic approach}
In this section, we give another proof of our main result using a variation of Riis forcing, presented in a form of model theoretic forcing in \cite{partiallydefinableforcing}.
More concretely, we prove that a variation of the construction presented in \S4.2 of \cite{partiallydefinableforcing} gives a model of $T^{1}_{2}(R)+UCP(\Delta^{b}_{1}(R))+\lnot injPHP^{n+1}_{n}(R)$. 
The challenging part here is to prove that our model satisfies $UCP(\Delta^{b}_{1}(R))$.
We reduce this problem to \cite{Razborov} by a variation of the notion of \textit{PHP-arrays} in \cite{Mykyta'smodels} and \cite{UCPvsinjPHPrevised}.
We assume the reader is familiar with \S 1 - \S 3, \S4.1, and \S4.2 of \cite{partiallydefinableforcing}. 

We work in the setting of \S4.2, that is, we fix:
\begin{itemize}
 \item A countable language $L$ containing $\{+,\cdot, 0, 1, <, lh(x), (x)_{y}\}$, where $lh(x)$ and $(x)_{y}$ are for sequence conding such that in the standard model every finite sequence is of the form $((n)_{0},...,(n)_{lh(n)-1})$.
 \item A countable $L$-structure $M$ that is a proper elementary extension of an $L$-expansion of the standard model $(\NN,+,\cdot, 0,1,<,lh(\cdot),(\cdot)_{\cdot})$.
 \item $L^{*}:=L\cup \{R\}$ for a new binary relation symbol $R \not \in L$. 
 \item Non-standard numbers $b_{0}, n$ such that $b_{0} <^{M} n^{o(1)}$ and $|{n}|<^{M} b^{o(1)}_{0}$, that is, $b_{0}^{k} <^{M} n$ and $|{n}|^{k} <^{M} b_{0}$ for any standard $k \in \NN$.
 \end{itemize}
 
 Intuitively, $b_{0}$ is a ``small'' number: 
 
 \begin{defn}[Definition 4.5 of \cite{partiallydefinableforcing}]
 A definable set $X$ over $M$ is \deff{small} if it is empty or there are $l \in \NN$ and an $L(M)$-definable surjection from $[b_{0}]^{l}$ onto $X$.
 \end{defn}
 
 In \S4.2 of \cite{partiallydefinableforcing}, the set of (codes of) small bijections from subsets of $M$ onto subsets of $[n]$ was considered as the poset of a forcing frame. 
 Here, we consider the set $\mathscr{P}$ of the codes of small \textit{partial injections} from subsets of $[n+1]$ into $[n]$; for each code $p \in \mathscr{P}$ of a small injection, we write:
 \begin{itemize}
  \item $p_{bij}$ for the bijection part, that is, the collection of all the two-sets of the partial injection represented by $p$.
  \item $p_{sing}$ for the empty-hole part, that is, the collection of all the singletons of the partial injection represented by $p$.
   \item $\dom(p)$ denotes the domain of $p$, that is, the domain of the bijection $p_{bij}$.
   Note that $\dom(p) \subseteq [n+1]$.
   \item $\ran(p)$ denotes the range of $p$, that is, the union of the range of $p_{bij}$ and all the singletons $p_{sing}$.
   Note that $\ran(p) \subseteq [n]$.
  \end{itemize} 
  
  We give a forcing frame (for the definition, see \S2.2 of \cite{partiallydefinableforcing}) as follows:
  \begin{defn}
 We set a countable forcing frame $\mathcal{F}:=(P,\leq,D_{0},D_{1},\ldots)$ as follows:
  \begin{enumerate}
  \item $P:=\mathscr{P}$.
  \item For $p,q \in \mathscr{P}$, we set $p \leq q$ iff $p \supseteq q$ as partial injections.
  \item $D_{0},D_{1},\ldots$ is an arbitrary enumeration of the sets $\{p \mid a \in \dom(p)\}$, $\{p \mid c \in \ran(p)\}$ for $a \in [n+1]$ and $c \in [n]$.
  \end{enumerate}
  \end{defn} 
  
  Furthermore, we consider the following forcing relation:
  \begin{defn}
  For $p \in \mathscr{P}$ and an $L(M)$-formula $\varphi$, we define the forcing relation $p \Vdash \varphi$ as a conservative (cf. \cite{partiallydefinableforcing} \S 2.6) universal pre-forcing (\cite{partiallydefinableforcing} Definition 2.3) induced by:
  \[p \Vdash R(s,t) :\Longleftrightarrow (s^{M},t^{M}) \in p_{bij}.\]
  \end{defn}
  
  \begin{rmk}
  As noted for $\Vdash_{Ri}$ in \cite{partiallydefinableforcing} \S 4.2, $\Vdash$ here is also a \textit{forcing} in the sense of \cite{partiallydefinableforcing} Definition 2.5.
  Indeed, since $\Vdash$ is conservative, it trivially satisfies Extension and Stability for $L(M)$-atoms.
  As for the remaining $L^{*}(M)$-atoms, that is, the atoms of the form $R(s,t)$, 
  \begin{itemize}
   \item If $p \Vdash R(s,t)$ and $q \leq p$, then $(s^{M},t^{M}) \in p_{bij} \subseteq q_{bij}$, and therefore $q \Vdash R(s,t)$. This establishes Extension.
   \item Suppose $[R(s,t)]$ is dense below $p$, that is, for any $p' \leq p$, there exists $q \leq p'$ such that $q \Vdash R(s,t)$.
   We show $p \Vdash R(s,t)$ holds. 
   First, we have $q \leq p$ such that $q \Vdash R(s,t)$.
   If $(s^{M},t^{M}) \not \in p_{bij}$, then $s^{M} \not \in \dom(p)$; otherwise $q$ cannot be an extension of $p$.
   Furthermore, since $p$ represents a small partial injection, and $b_{0}<n^{o(1)}$, $[n]\setminus \ran(p)$ is infinite, and we can take $h \in [n]\setminus (\ran(p)\cup\{t^{M}\})$.
   Then we obtain another extension $r \leq p$ such that $(s^{M},h) \in r_{bij}$.
   By assumption of density, there exists $q' \leq r$ such that $q' \Vdash R(s,t)$, which is absurd since $q$ maps $s^{M}$ to $h \neq t^{M}$.
   This establishes Stability. 
  \end{itemize}
  \end{rmk}
  
  By the same reason as Lemma 4.6 of \cite{partiallydefinableforcing}, $\Vdash$ is definable (cf. Definition 3.3 of \cite{partiallydefinableforcing}) for a kind of ``sharply-bounded formulae'':
  
\begin{defn}
 For $b_{0} \in M$, $\Delta^{b_{0}}_{0}(R)$ denotes the closure of the set of quantifier-free $L^{*}(M)$-formulas by $b_{0}$-bounded quantification, i.e. $\exists x < b_{0}$ and $\forall x < b_{0}$.
 If we additionally allow unrestricted existential quantification, we get the set $\Sigma^{b_{0}}_{1}(R)$.
\end{defn}

\begin{lemma}
$\Vdash$ is definable for all the $\Delta^{b_{0}}(R)$-formulae.
\end{lemma}

\begin{proof}
Let $p \in \mathscr{P}$ and $\varphi(\bar{x}) \in \Delta^{b_{0}}_{0}(R)$, where $\bar{x}=(x_{1},\ldots,x_{l})$.
We show that $\{\bar{a} \mid p || \varphi(\bar{a})\}$ is an $L(M)$-definable set over $M$.
For each term $t(\bar{x},\bar{y})=t(x_{1},\ldots,x_{l},y_{1},\ldots, y_{m})$ occurring in $\varphi$, where $\bar{x}$ and $\bar{y}$ exhaust all the variables occurring in $t$, set
\[I_{t}(\bar{a}) := \{t(\bar{a},\bar{c}) \mid c_{1},\ldots, c_{m}< b_{0}\}\] 
for each $\bar{a}=(a_{1},\ldots,a_{l}) \in M^{l}$.
Note that $I_{t}(\bar{a})$ is small.
Furthermore, for each $a_{1},\ldots,a_{l} \in M$, let $T_{\varphi(\bar{x})}(a_{1},\ldots,a_{l})$ be the union of all the $I_{t}(\bar{a})$ for all the terms occurring in $\varphi(\bar{x})$.
Since the number of terms occurring in $\varphi(\bar{x})$ is finite, $T(\bar{a})$ is again small.

Now, we use Lemma 3.10-3.13 in \cite{partiallydefinableforcing}; for each $\Delta^{b_{0}}_{0}(R)$-formula $\psi(\bar{z})$, we construct uniformly definable maximal antichains $X_{\psi(\bar{z}),\bar{d}}$ for each $[\psi(\bar{d})]$, with which we can define
\[p||\psi(\bar{d}) \Longleftrightarrow \exists q \in X_{\psi,\bar{d}}.\ p||q.\]
Note that $p||q$ is definable since it is equivalent to: 
\begin{itemize}
 \item for each $\alpha \in \dom(p) \cap \dom(q)$, $p$ and $q$ map $\alpha$ to the same hole, and
 \item for each $\beta \in \ran(p) \cap \ran(q)$, $p$ and $q$ both have the singleton $\{\beta\}$ or both match $\beta$ with the same pigeon.
\end{itemize}

Let $X_{\psi(\bar{z}),\bar{d}}$ be the collection of all the minimal partial injections covering $T_{\psi(\bar{z})}(\bar{d})$, that is, the collection of partial injections $p$ such that:
\begin{enumerate}
 \item For each $\langle \alpha,\beta \rangle \in p_{bij}$, $\alpha \in T_{\psi(\bar{z})}(\bar{d})$ or $\beta \in T_{\psi(\bar{z})}(\bar{d})$.
 \item For each $\{\beta\} \in p_{sing}$, $\beta \in T_{\psi(\bar{z})}(\bar{d})$.
 \item For each $\gamma \in T_{\psi(\bar{z})}(\bar{d})$, 
 \begin{itemize}
  \item a pair of the form $\langle \gamma,\beta \rangle$ is in $p$, and
  \item a pair of the form $\langle \alpha, \gamma \rangle$ or a single $\{\gamma\}$ is in $p$ if $\gamma \in [n]$.
  \end{itemize}
\end{enumerate}

We observe the following:
\begin{enumerate}
 \item each condition in $X_{\psi(\bar{z}),\bar{d}}$ is small, and therefore each $X_{\psi(\bar{z}),\bar{d}}$ is a maximal antichain of $\mathscr{P}$. 
 \item for each $\psi(\bar{z}) \in \Delta^{b_{0}}_{0}(R)$, $\{(p,\bar{d}) \mid p \in X_{\psi(\bar{z}),\bar{d}}\}$ is definable.
 \item for each $\psi(\bar{z})\in \Delta^{b_{0}}_{0}(R)$, $X_{\psi(\bar{z}),\bar{d}} = X_{\lnot\psi(\bar{z}),\bar{d}}$.
 \item for each $\psi(\bar{z}), \theta(\bar{z}) \in \Delta^{b_{0}}_{0}(R)$, $X_{\psi(\bar{z})\land \theta(\bar{z}),\bar{d}}$ refines $X_{\psi(\bar{z}),\bar{d}} = X_{\lnot\psi(\bar{z}),\bar{d}}$ and $X_{\theta(\bar{z}),\bar{d}} = X_{\lnot\theta(\bar{z}),\bar{d}}$.
 \item for each $\psi(w,\bar{z})$, $X_{\exists w < b_{0}.\ \psi(w,\bar{z})}$ refines 
 \[\bigcup_{e<^{M}b_{0}}X_{\psi(w,\bar{z}),(e,\bar{d})} = \bigcup_{e<^{M}b_{0}}X_{\lnot\psi(w,\bar{z}),(e,\bar{d})}.\]
 \end{enumerate}

Therefore, by \cite{partiallydefinableforcing} Lemma 3.13, each $X_{\psi(\bar{z}),\bar{d}}$ is a maximal antichain of $[\psi(\bar{d})]$, with which we can define $p || \varphi$ by \cite{partiallydefinableforcing} Lemma 3.10.
\end{proof}

Now, by the very same argument as the proof of Theorem 4.3 of \cite{partiallydefinableforcing}, we obtain the following:
\begin{prop}
Let $G$ be a generic filter of $\mathscr{P}$.
Then $M[G]$ induced by the forcing frame $\mathcal{F}$ is an $L^{*}$-expansion $(M,R^{M})$ of $M$ such that:
\begin{enumerate}
 \item $R^{M}$ codes an injection from the whole universe $M$ into $[n]$.
 \item $M[G]$ satisfies the least number principle for $\Sigma^{b_{0}}_{1}(R)$.
\end{enumerate}
\end{prop}

As stated in the \cite{partiallydefinableforcing} Remark 4.4, we can extract a model of $T^{1}_{2}(R)$ from $M[G]$ in the previous Proposition:
\begin{cor}\label{T12withextremeinjection}
 The cut $I$ of $M[G]$ generated by $\{2^{|{n}|^{k}} \mid k \in \NN\}$ is a model of $T^{1}_{2}(R)$.
\end{cor}

The main result in this section is the following:
\begin{thm}\label{UCPvsinjPHPmodelsolution}
The model $I$ in the Corollary \ref{T12withextremeinjection} satisfies $\forall l,d,k.\ UCP^{l,d}_{k}(\Delta^{b}_{1}(R))$.
\end{thm}

Towards the goal, we prepare two lemmas:
\begin{lemma}\label{shallowDTdecidestheta}
Let $\theta(x_{1},\ldots, x_{k}) \in \Delta^{b}_{1}(R)$, where $\bar{x}=(x_{1},\ldots,x_{k})$ exhausts all the free variables occurring in $\theta$, and let $a_{1}, \ldots, a_{k} \in I$.
Then there exists a $\{0,1\}$-labeled $injPHP$-tree $T_{\theta(\bar{x}),\bar{a}}$ over $([n+1],[n])$ of height $\leq |{n}|^{O(1)} < b_{0}$ such that:
\begin{itemize}
 \item For each $p \in br_{1}(T_{\theta(\bar{x}),\bar{a}})$, $p \Vdash \theta(\bar{a})$.
 \item For each $p \in br_{0}(T_{\theta(\bar{x}),\bar{a}})$, $p \Vdash \lnot\theta(\bar{a})$.
\end{itemize}  
Furthermore, $\{T_{\theta(\bar{x}),\bar{a}}\}_{\bar{a}}$ can be chosen so that it is a uniformly definable family of definable sets in $M$.
\end{lemma}

\begin{proof}
Since $\theta$ is $\Delta^{b}_{1}(R)$, there exists a constant $C>0$ such that $\theta(\bar{a})$ can be decided by an polytime oracle Turing machine with $|{n}|^{C}$-many queries of the form ``$R(i,j)$ is true?'' for any inputs $\bar{a} \in I$.
We may assume that $i \in [n+1]$ and $j \in [n]$ by fixing the answer to ``no'' for ``$R(i,j)$ is true?'' for $i \not \in [n+1]$ or $j \not \in [n]$. 
Therefore, in $M$, $\theta(\bar{a})$ is decided by a binary tree $B$ of height $\leq |{n}|^{C}$.
Let $T_{\theta(\bar{x}),\bar{a}}$ be an $injPHP$-tree over $([n+1],[n])$ asking the pigeons and holes in queries in $B$ along $B$.
The height of $T_{\theta(\bar{x}),\bar{a}}$ is at most $\leq 2|{n}|^{C}$, and each branch $b \in T_{\theta(\bar{x}),\bar{a}}$ induces a branch $b'$ of $B$.
We label $b$ with $1$ if $\theta(\bar{a})$ is true along $b'$ and with $0$ otherwise.
Then $T_{\theta(\bar{x}),\bar{a}}$ is a desired $\{0,1\}$-labeled $injPHP$-tree.
Indeed, for $p \in br_{1}(T_{\theta(\bar{x}),\bar{a}})$, consider the induced branch $b'$ of $B$.
We have that, for each query ``$R(i,j)$ is true?'' asked along $b'$, 
\begin{itemize}
 \item if the answer in $b'$ is ``yes,'' $p \Vdash R(i,j)$.
 \item if the answer in $b'$ is ``no,'' $p \Vdash \lnot R(i,j)$.
\end{itemize}
Therefore, for any generic filter $G$ including $p$, $\theta(\bar{a})$ is true because the branch $b'$ of $B$ is satisfied by $R^{M}$. 

\end{proof}

\begin{lemma}\label{refiningtreelemma}
Let $A$ be an antichain of partial injections from $[n+1]$ into $[n]$ of size $\leq s$ in a standard model $\NN$.
Then there exists a $\{0,1\}$-labeled $injPHP$-tree $T$ over $([n+1],[n])$ of height $\leq 2s^{2}$ refining $A$, that is:
\begin{itemize}
 \item for each $b \in br_{1}(T)$, there exists $p \in A$ such that $p \subseteq b$ as partial injections.
 \item for each $b \in br_{0}(T)$, $b \perp p$ for arbitrary $p \in A$ as partial injections.
\end{itemize}
\end{lemma}

\begin{rmk}
The statement is arithmetic and we have that it holds also in our ground model $M$.
\end{rmk}

\begin{proof}
 By induction on $s$.
 When $s=0$, $A=\{\emptyset\}$ or $A=\emptyset$, and the tree of height $0$ (that is, the root only) with the label $1$ or $0$ respectively suffices.

 Assume the statement is true for $s$.
 Let $A$ be an antichain of partial injections of size $\leq s+1$.
 We may assume $A\neq \emptyset$ by the case of $s=0$, and we take $p \in A$.
 Consider an $injPHP$-tree $U$ querying all the pigeons in $\dom(p)$ and all the holes in $\ran(p)$.
 The height of $U$ is at most $\leq 2(s+1)$.
 For each branch $c$ of $U$, consider 
 \[A^{c} := \{q\setminus c  \mid q \in A \ \& \ q || c\}.\]
 Since $A$ is an antichain, for any $q \in A$, $\dom(q) \cap \dom(p) \neq \emptyset$ or $\ran(q) \cap \ran(p) \neq \emptyset$.
 Therefore, each $A^{c}$ is an antichain of partial injections of size $\leq s$, and there exists an $\{0,1\}$-labeled tree $T_{c}$ refining $A^{c}$.
 Let 
 \[T:=U*\sum_{c \in br(U)}T_{c}^{c}.\] 
 Then $T$ is a desired $\{0,1\}$-labeled $injPHP$-tree.
 Indeed, 
 \[height(T) \leq 2s^{2}+2(s+1) \leq 2(s+1)^{2}.\]
 Furthermore, for a branch $b \in br_{1}(T)$, consider the decomposition $b=cd^{c}$, where $c \in br(U)$ and $d \in br_{1}(T_{c})$. 
 Since $d \in br_{1}(T_{c})$ and $T_{c}$ is a refinement of $A^{c}$, there exists $r \in A$ such that $r^{c} \subseteq d$. 
 Therefore,
 \[r \subseteq cd^{c} = b.\]
 
 Moreover, for a branch $b \in br_{0}(T)$, consider the decomposition $b=cd^{c}$, where $c \in br(U)$ and $d \in br_{0}(T_{c})$.
 For any $r \in A$ with $r||c$, $r^{c} \in A^{c}$, and $r^{c} \perp d$.
 Therefore, $r \perp cd^{c}=b$. 

\end{proof}

\begin{proof}[Proof of Theorem \ref{UCPvsinjPHPmodelsolution}]
Suppose there exists $l,d,k \in I$ and a $\Delta^{b}_{1}(R)$-formula $\theta$ such that 
\[I \models \lnot UCP^{l,d}_{k}(\theta).\]
$UCP^{l,d}_{k}(\theta)$ is a bounded formula, and $I$ is a cut of $M[G]$, therefore, by the absoluteness, it follows that
\[M[G] \models \lnot UCP^{l,d}_{k}(\theta).\] 
By Truth Lemma (\cite{partiallydefinableforcing} Lemma 2.19), there exists $p \in G$ such that $p \Vdash \lnot UCP^{l,d}_{k}(\theta)$.
Note that $\Vdash$ is closed under logical consequences (cf. \cite{partiallydefinableforcing} Corollary 2.20.(2)), and therefore we have the following:
\begin{enumerate}
\item\label{emptyord-set} $p \Vdash \forall i \in [l] \left( (\forall j' \in [d] \forall e' \in [m].\ \lnot \theta(i,j',e')) \lor (\forall j \in [d] \exists e \in [m].\theta(i,j,e))\right)$.
Therefore, for each $i \in [l]$ and $q \leq p$, either one of the following holds:
\begin{enumerate}
 \item\label{emptycondition} there exists $r \leq q$ such that $\forall j' \in [d] \forall e' \in [m].\ r \Vdash \lnot \theta(i,j',e')$.
 \item\label{d-setcondition} there exists $r \leq q$ such that $\forall j \in [d]\exists e \in [m] \exists s \leq r.\ s \Vdash \theta(i,j,e)$.
\end{enumerate}
\item\label{residentisunique} $p \Vdash \forall i \in [l] \forall j \in [d] \forall e \neq e' \in [m].\ (\lnot \theta(i,j,e) \lor \lnot \theta(i,j,e'))$.
In other words, for each $i \in [l], j \in [d]$, and distinct $e,e' \in [m]$, no extension $q \leq p$ can force both $\theta(i,j,e)$ and $\theta(i,j,e')$ simultaneously.
\item\label{roomisunique} $p \Vdash \forall (i,j) \neq (i',j') \in [l]\times[d] \forall e \in [m].\ (\lnot \theta(i,j,e) \lor \lnot \theta(i',j',e))$.
In other words, for any distinct pairs $(i,j), (i',j') \in [l]\times[d]$ and $e \in [m]$, no extension $q \leq p$ can force both $\theta(i,j,e)$ and $\theta(i',j',e)$ simultaneously.
\item\label{roomtotal} $p \Vdash \forall e \in [m] \exists (i,j) \in [l] \times [d].\ \theta(i,j,e)$.
Therefore, for each $e \in [m]$ and $q \leq p$, there always exists some $r \leq q$ and $(i,j) \in [l] \times [d]$ such that $r \Vdash \theta(i,j,e)$.
\end{enumerate}

Now, work in $M$. (Recall that $M \models Th(\NN)$.)
By Lemma \ref{shallowDTdecidestheta}, we have a code of a family $\{T_{\theta(\iota,\kappa,\epsilon),(i,j,e)}\}_{i \in [l], j\in [d], e \in [m]}$ in $M$.
Let 
\[A_{i,j,e} := br_{1}(T_{\theta(\iota,\kappa,\epsilon),(i,j,e)}^{p}),\]
 regarded as a set of partial injections.
We have that each $A_{i,j,e}$ is an antichain of $\mathscr{P}$, and $height(A_{i,j,e}) \leq |{n}|^{C}$ for some constant $C>0$.

For each $e \in [m]$, $\bigcup_{(i,j) \in [l] \times [d]} A_{i,j,e}$ is again an antichain by (\ref{roomisunique}).
Therefore, applying Lemma \ref{refiningtreelemma} in $M$, we obtain a family $\{S_{e}\}_{e \in [m]}$ of $\{0,1\}$-labeled $injPHP$-trees such that each $S_{e}$ refines $\bigcup_{(i,j) \in [l] \times [d]} A_{i,j,e}$.
Without loss of generality, we may assume that $S_{e}^{p}=S_{e}$.
Note that $height(S_{e}) \leq 2|{n}|^{2C} < b_{0}$.
Moreover, for each $b \in br(S_{e_{0}})$, there uniquely exists $(i,j) \in [l]\times[d]$ and $s \in A_{i,j,e_{0}}$ such that $s \subseteq b$ as partial injections.
On the other hand, by (\ref{roomtotal}), there exist $r \leq bp$ and $(i_{0},j_{0}) \in [l] \times [d]$ such that $r \Vdash \theta(i_{0},j_{0},e_{0})$.
$T_{\theta(\iota,\kappa,\epsilon),(i_{0},j_{0},e_{0})}$ is of height $<b_{0}$, and $r$ is also of size $<b_{0}$, and $b_{0} < n^{o(1)}$. 
Hence, there exists a branch $a$ of $T_{\theta(\iota,\kappa,\epsilon),(i_{0},j_{0},e_{0})}$ such that $a || r$.
Since $r \Vdash \theta(i_{0},j_{0},e_{0})$, we have $a \in br_{1}(T_{\theta(\iota,\kappa,\epsilon),(i_{0},j_{0},e_{0})})=A_{i_{0},j_{0},e_{0}}$.
This implies $b || a \in \bigcup_{(i,j) \in [l] \times [d]} A_{i,j,e_{0}}$, and therefore, by the definition of $S_{e_{0}}$, $b$ is an extension of some partial injection in some $A_{i,j,e_{0}}$ ($(i,j) \in [l] \times [d]$).
Since $\bigcup_{(i,j) \in [l] \times [d]} A_{i,j,e_{0}}$ is an antichain, such $(i,j)$ is unique, and we denote them by $(i_{b},j_{b})$.
Let $\widetilde{S}_{e}$ be the $[l]\times[d]$-labeled $injPHP$-tree obtained by relabeling each branch $b$ of $S_{e}$ by $\langle i_{b},j_{b} \rangle$.

Similarly, for each $(i,j) \in [l] \times [d]$, $\bigcup_{e \in [m]} A_{i,j,e}$ is again an antichain by (\ref{residentisunique}), and we obtain a family $\{V_{i,j}\}_{(i,j) \in [l] \times [d]}$ of $\{0,1\}$-labeled $injPHP$-trees of height $\leq 2|{n}|^{2C}<b_{0}$ such that each $V_{i,j}$ refines $\bigcup_{e \in [m]} A_{i,j,e}$.
Without loss of generality, we may assume that $V_{i,j}^{p}=V_{i,j}$.
For each $b \in br(V_{i,j})$, exactly one of the following holds:
\begin{itemize}
\item $b \perp a$ for arbitrary $a \in \bigcup_{e \in [m]} A_{i,j,e}$.
In this case, applying (\ref{emptyord-set}) to $q \leq bp$, we see that the case (\ref{d-setcondition}) cannot happen, and (\ref{emptycondition}) holds.
It implies $b \perp a$ for arbitrary $a \in \bigcup_{j' \in [d], e' \in [m]} A_{i,j',e'}$.
Indeed, if $b || a$ holds for some $a \in A_{i,j',e'}$, (\ref{emptycondition}) holds for $q = abp$, which contradicts $ap \Vdash \theta(i,j',e')$.

\item $b \supseteq a$ as partial injections for some $e \in [m]$ and $a \in A_{i,j,e}$. 
In this case, $e$ is unique since $\bigcup_{e \in [m]} A_{i,j,e}$ is an antichain, and we denote it by $e_{b}$.
\end{itemize}

Let $\widetilde{V}_{i,j}$ be $([m] \sqcup\{\perp\})$-labelled $injPHP$-tree obtained by relabeling each branch $b$ of $V_{i,j}$ by: 
\begin{itemize}
 \item the symbol ``$\perp$'' if $b \perp a$ for arbitrary $a \in \bigcup_{e \in [m]} A_{i,j,e}$.
 \item the number $e_{b} \in [m]$ otherwise.
\end{itemize}  

For $(i,j) \in [l] \times [d]$, we define a $([m] \sqcup\{\perp\})$-labelled $injPHP$-tree $\widetilde{U}_{i,j}$ by:
\[\widetilde{U}_{i,j} := \widetilde{V}_{i,1} * \sum_{b \in br(\widetilde{V}_{i,1})\setminus br_{\perp}(\widetilde{V}_{i,1})} \widetilde{V}_{i,j}^{b}.\]
% $\widetilde{V}_{i,1}$ concatenated by $\widetilde{V}_{i,j}^{b}$ for each $b \in br(\widetilde{V}_{i,1})\setminus br_{\perp}(\widetilde{V}_{i,1})$.

Then each $\widetilde{U}_{i,j}$ again refines $\bigcup_{e \in [m]} A_{i,j,e}$.
Furthermore, each branch $d \in br(\widetilde{U}_{i,j})$ satisfies exactly one of the following:
\begin{itemize}
 \item $d \in br_{\perp}(\widetilde{V}_{i,1})$ and $d \perp a$ for arbitrary $a \in \bigcup_{j' \in [d], e' \in [m]} A_{i,j',e'}$.
 \item $d$ is of the form $d=bc^{b}$, where $b \in br_{e}(\widetilde{V}_{i,1})$ for some $e \in [m]$, and $c \in br_{e'}(\widetilde{V}_{i,j})$ for some $e' \in [m]$.
 We define $\widetilde{e}_{b}:=e'$.
\end{itemize}

Now, we define $[l]\times[d]\times[m]$-labeled $injPHP$-trees $(X_{i,j})_{i \in [l], j \in [d]}$ and $(Y_{e})_{e \in [m]}$ as follows:
\begin{itemize}
 \item $X_{e} := \widetilde{S}_{e}*\sum_{b \in br(\widetilde{S}_{e})} \widetilde{V}_{i_{b},j_{b}}^{b}$, and each branch $bc^{b}$ ($b \in br(\widetilde{S}_{e})$ and $c \in \widetilde{V}_{i_{b},j_{b}}$) is labeled by $\langle i_{b},j_{b}, e\rangle$.
\item $Y_{i,j} := \widetilde{V}_{i,j} * \sum_{c \in br(\widetilde{V}_{i,j}) \setminus br_{\perp}(\widetilde{V}_{i,j})} \widetilde{S}_{e_{c}}^{c}$, and each branch $cb^{c}$ ($c \in br_{e}(\widetilde{V}_{i,j})$ for some $e \in [m]$ and $b \in \widetilde{S}_{e_{c}}$) is labeled by $\langle i,j, e_{c}\rangle$, and each branch $c \in br_{\perp}(\widetilde{V}_{i,j})$ is again labeled by $\perp$.
\end{itemize}

 We see that $(X_{i,j})_{i \in [l],j \in [d]}$ and $(Y_{e})_{e \in [k]}$ satisfy the following:
\begin{itemize}
 \item For each $i$，$br_{\perp}(X_{i,1})= \cdots = br_{\perp}(X_{i,d}) = br_{1}(N_{i})$.
  \item For each $i,j,e$, $br_{\langle i,j,e \rangle}(X_{i,j})= br_{\langle i,j,e \rangle}(Y_{e})$ (as sets of partial injections). 
\end{itemize}

Now, we are ready to repeat the construction of a low-degree NS-refutation of $injPHP^{n+1}_{n}$ over the ring $\ZZ_{d}$ given in the proof of Theorem 53 in \cite{UCPvsinjPHPrevised} in $M$, which contradicts Razborov's linear degree lower bound for $injPHP^{n+1}_{n}$ (\cite{Razborov}) in $M$. (See also \cite{UCPvsinjPHPrevised} Corollary 25, and the statements necessary for this proof are all arithmetical and therefore hold in $M$.)

\end{proof}

\begin{rmk}\label{mostweakPHPisindependent}
With a little modification of the previous proof, actually we can show that 
\[T^{1}_{2}(R) + UCP^{l,d}_{k}(\Delta^{b}_{1}(R)) \not\vdash injPHP^{univ}_{n}(R),\]
 where $injPHP^{univ}_{n}(R)$ states that ``$R$ cannot code an injection from the whole universe into $[n]$.''
It suffices to:
\begin{itemize}
 \item change the definition of $\mathscr{P}$ to ``the set of all the small partial injections from $M$ to $[n]$'' and include $\{p \mid a \in \dom(p)\}$ for all $a \in M$ in $D_{0},D_{1},\ldots$, which establishes $M[G] \models \lnot injPHP^{univ}_{n}(R)$.
 \item consider shallow $injPHP$-trees over $([2^{|{n}|^{C}}],[n])$ for sufficiently large $C>0$ instead of over $([n+1],[n])$ for establishing $I \models UCP^{l,d}_{k}(\theta)$. ($C$ depends on $\theta$.)
 \end{itemize}
 This also explains the difficulty to lift the result to higher base theories, say, $T^{2}_{2}(R)$ instead of $T^{1}_{2}(R)$; it is well-known that 
 $T^{2}_{2}(R) \vdash injPHP^{2n}_{n}(R)$ (\cite{T22provesWPHP}).
 
\end{rmk}

\begin{rmk}
 The approach in \S \ref{The Main Result} can be rephrased in the model-theoretic forcing approach in this section as \cite{partiallydefinableforcing} \S 4.3; a forcing frame of partial injections having ``large bijection parts'' are considered here, and the $o(n)$-evaluation for depth $0+\frac{1}{2}$-formulae and $1+\frac{1}{2}$-formulae each corresponds to give maximal antichains for sharply-bounded formulae and $\Sigma^{b_{1}}(R)$-formulae respectively. 
\end{rmk}

\section{Acknowledgement}
The author would like to express sincere thanks to Jan Kraj\'{i}\v{c}ek, Mykyta Narusevych, and Ond\v{r}ej Je\v{z}il for their hospitality and stimulating questions and comments during my visit to Prague in 2023 and 2024, which motivated this research.
The author is also grateful for Pavel Pudl\'{a}k, Neil Thapen, Erfan Khaniki, and Dimitrios Tsintsilidas for their comments and feedbacks. 
We also would like to thank Toshiyasu Arai for his teachings and guidance to proof theoretic techniques.

This research was supported by:
\begin{itemize}
\item FoPM, WINGS Program, the University of Tokyo, and
\item JSPS KAKENHI Grant Number 22KJ1121, Grant-in-Aid for JSPS Fellows.
\end{itemize}

 \end{document}